\documentclass[11pt,reqno]{amsart} 
\usepackage[utf8]{inputenc}
\usepackage{amsfonts,amsmath,amsthm,amssymb}
\usepackage[justification=centering]{caption}
\usepackage{color}
\usepackage{fullpage}
\usepackage{enumitem}   
\usepackage{hyperref,float}
\usepackage{mathtools}
\usepackage{xcolor}
\usepackage{thm-restate}
\usepackage{thmtools}
\usepackage{tikz}

\newtheorem{theorem}{Theorem}[section]
\newtheorem{proposition}[theorem]{Proposition}
\newtheorem{lemma}[theorem]{Lemma}
\newtheorem{corollary}[theorem]{Corollary}

\newtheorem{conjecture}[theorem]{Conjecture}
\theoremstyle{definition}
\newtheorem{example}[theorem]{Example}
\newtheorem{definition}[theorem]{Definition}

\usepackage{etoolbox}
\AtEndEnvironment{proof}{\setcounter{claim}{0}}
\newtheoremstyle{notation}
{\topsep}
{\topsep}
{}
{}
{\bfseries}
{.}
{.5em} 
{} 
\theoremstyle{notation}

\newcommand{\Des}{\text{Des}}
\newcommand{\Asc}{\text{Asc}}

\newcommand{\calR}{\mathcal{R}}
\newcommand{\calB}{\mathcal{B}}
\newcommand{\calC}{\mathcal{C}}
\newcommand{\Sym}{\mathfrak{S}}

\newcommand{\N}{\mathbb{N}}
\newcommand{\JEem}{\textit}

\begin{document}

\title[On Graphs of Sets of Reduced Words]{On Graphs of Sets of Reduced Words}
\author{Jennifer Elder} 
\address{Department of Computer Science, Math \& Physics. Missouri Western State University, 4525 Downs Drive, St. Joseph, MO 64507 }
\email{\textcolor{blue}{\href{mailto:jelder8@missouriwestern.edu}{jelder8@missouriwestern.edu}}}


\begin{abstract}
Any permutation in the finite symmetric group can be written as a product of simple transpositions $s_i = (i~i+1)$. For a fixed permutation $\sigma \in \mathfrak{S}_n$ the products of minimal length are called reduced decompositions or reduced words, and the collection of all such reduced words is denoted $\mathcal{R}(\sigma)$. Any reduced word of $\sigma$ can be transformed into any other by a sequence of commutation moves or long braid moves. One area of interest in these sets are the congruence classes defined by using only braid or only commutation relations. The set $\mathcal{R}(\sigma)$ can be drawn as a graph, $G(\sigma)$, where the vertices are the reduced words, and the edges denote the presence of a commutation or braid move between the words. This paper presents new work on subgraph structures in $G(\sigma)$, as well as new formulas to count the number of braid edges and commutation edges in $G(\sigma)$. We also include work on bounds for the number of braid and commutation classes in $\mathcal{R}(\sigma)$.
\end{abstract}

\maketitle

\section{Introduction}

Permutations are extremely useful and complex combinatorial objects, which has led to extensive study of their properties. This study spans general permutation statistics \cite{elder2023homomesies, FindStat}, which include patterns \cite{TENNER2007888, T1} or descent and ascent properties \cite{Billey_Permu, DIAZLOPEZ2021112375}, to decomposition into cycles and related structures \cite{R, S1, T2}. Specifically, the decomposition of permutations into products of cycles called \emph{simple transpositions}. 

\begin{definition}\label{def:sym_group}
    Let $[n] := \{1,2, \ldots n \}$, and denote the symmetric group on this set as $\mathfrak{S}_n$. A permutation $\sigma\in \Sym_n$ is written in \JEem{one line notation} as $\sigma = [\sigma_1 ~\sigma_2 ~ \ldots ~\sigma_n]$, where $\sigma(i) = \sigma_i$.

    We define the \JEem{descent set of} $\sigma$ as $\Des(\sigma ) = \{ i\in [n] \mid \sigma_i>\sigma_{i+1}\}$
    
    The group $\Sym_n$ can be generated by a set of functions called \textit{simple transpositions}: $s_i = (i~i+1)$ for $1\leq i\leq n-1$. 
  
These generators also have specific relationships:
\begin{align}
s_i^2 &= e, \\
s_i s_j &= s_j s_i, \label{eq:1} \mbox{ for } |i-j|>1 ,\\
s_i s_{i+1} s_i &= s_{i+1} s_i s_{i+1}, \label{eq:2} \mbox{ for } |i-j|=1.
\end{align}
\end{definition}

The second relation, labeled \eqref{eq:1}, is called a \JEem{commutation relation}, while the third relation, labeled \eqref{eq:2}, is called a \JEem{braid relation}. 

A minimal sequence of generators that produces $\sigma \in \mathfrak{S}_n$ is called a \JEem{reduced decomposition}. 
For example, in $\mathfrak{S}_5$, the permutation $\sigma =[42315]$ can be written as the product of length 5: $s_1s_2s_3s_2s_1= (1~2)(2~3)(3~4)(2~3)(1~2)$. For a permutation $\sigma$, we say that $s_j$ is in the \JEem{support of $\sigma$} if $s_j$ is a present in a reduced decomposition of $\sigma$. For example, $s_1$ is in the support of $[42315]$.

For simplicity, we often write $\sigma$ as a \textit{reduced word} created from the indices of a reduced decomposition of $\sigma$. In the previous example, we would produce the reduced word 12321. 

Once we have a reduced word, we wish to find all reduced words for $\sigma$. We denote the set of all reduced words as $\calR(\sigma)$
The letters of these words obey commutation relations and braid relations. These relations between reduced words, structures found in the sets $\calR (\sigma)$ and the congruence classes they produce, have been studied extensively \cite{RR, S1, T2}. As the sets become large, they become harder to study simply as sets. In the previous example, there are six reduced words for $\sigma$. This is why we create \JEem{graphs of sets of reduced words}. For example,

Figure~\ref{fig:graphofsigma} shows the graph $G([42315])$. 

\begin{figure}[H]
\begin{center}
\begin{tikzpicture}[node distance=2cm]
\node(v1)	 			            	{$12321$};
\node(v2)       [right of=v1] 		    {$13231$};
\node(v3)      [below right of=v2]  	{$31231$};
\node(v4)      [above right of=v2]  	{$13213$};
\node(v5)      [below right of=v4]  	{$31213$};
\node(v6)      [right of=v5]  	        {$32123$};
\node(v0)      [left of =v1]            {$G([42315])=$};

\draw[dashed] (v1) --(v2);
\draw (v2) --(v3);
\draw (v2) --(v4);
\draw (v5) --(v4);
\draw (v5) --(v3);
\draw[dashed] (v6) --(v5);
\end{tikzpicture}
\vspace{0.1in}
\caption{The graph $G([42315])$, where solid edges denote commutation relations, and dashed edges denote braid relations.}
\label{fig:graphofsigma}
\end{center}
\end{figure}

Studying properties of graphs of sets of reduced words is a way to further understand the structure of $\calR(\sigma)$. For example, if you are interested in better understanding the congruence classes in $\calR(\sigma)$ produced by the braid and commutation relations, the graphs are a good place to start. In Figure \ref{fig:graphofsigma}, you can see the braid classes, $\calB(\sigma)$, by deleting the solid edges and studying the remaining connected components.

Our motivating question is as follows: is there an upper bound on the ratio $|\calB(\sigma)|/|\calR(\sigma)|$ for arbitrary $\sigma$? 

As we pursued an answer to this question, we looked for results that would allow us to break our graphs into smaller pieces. This led to the development of new families of subgraphs for $G(\sigma)$, new methods of breaking $\sigma$ into smaller pieces based on its descent set, and bounds for the ratio for certain classes of permutations. 

In Section \ref{sec:background}, we give the standard definitions related to $\calR(\sigma)$, the weak order lattice, existing work on the size of $\calR(\sigma)$, and properties of $G(\sigma)$ that we use throughout the paper. In Section \ref{sec:notation}, we provide new definitions for when we will call two permutations $\sigma$ and $\tau$ equivalent, as well as new notation for specific families of equivalent permutations. Our results include the following:

\begin{enumerate}
    \item In Section \ref{sec:subgraphs_des}:
    \begin{enumerate}
        \item We create a subgraph partition of $G(\sigma)$ into the subgraphs related to $\sigma s_i$, for $i\in~\Des(\sigma)$, (Theorem \ref{graph_partition}).
        \item We produce new recursive formulas for the number of braid and commutation edges in $G(\sigma)$, (Corollaries \ref{recDEG} and \ref{coro_edge2}).
    \end{enumerate}
    \item In Section \ref{sec:shuffle_sub}:
    \begin{enumerate}
        \item We provide results on writing reduced decompositions of $\sigma$ given consecutive elements in $\Des(\sigma)$ (Lemma \ref{breakingPerms}).
        \item We characterization for the permutations with a fixed descent set with minimal length (Theorem \ref{thm:descentsetbreak}, Corollary \ref{cor:minimalperm}).
    \end{enumerate}
    \item In Section \ref{subsec:shuffles}:
    \begin{enumerate}
        \item We create a partition of $V(G(\sigma))$ into subsets related to $\tau^{-1}\sigma$, where $\tau \leq_W \sigma$, and $\tau^{-1}\sigma$ has a descent set containing only consecutive elements, (Theorem \ref{thm:spliting_perms}).
        \item We provide characterizations for when these graphs have ``nice" structures (Corollary \ref{cor:nice_shuffle_graphs}).
    \end{enumerate}
    \item In Section \ref{sec:motivation}:
    \begin{enumerate}
    \item We discuss the motivating problem that led to these families of subgraphs.
    \item We prove bounds for the ratio $|\calB(\sigma)|/|\calR(\sigma)|$ in the special case where $\sigma$ is the longest permutation, $w_0$ (Theorem \ref{weakerV}).
        \item We prove cases of the motivating problem using the subgraph structures (Lemma \ref{casesproved}, Theorem \ref{mainresult}).
    \end{enumerate}
\end{enumerate}
We conclude in Section \ref{sec:future} with a discussion of future directions for this research.

\section{Preliminaries} \label{sec:background}

In Definition \ref{def:sym_group}, we defined the symmetric group, descents, and the simple transpositions.  

\begin{definition}
    An \JEem{inversion} is any pair in the one line notation of $\sigma$ such that $\sigma_i>\sigma_j$ where $i<j$. The number of inversions of $\sigma$ is the length of $\sigma$, denoted $\ell(\sigma)$.

    The identity $e=[1 ~2~\ldots ~n]$ has the shortest length with zero inversions, while the permutation $w_0~=~[n~n-1~\ldots 2~1]$ has the maximum length with $\frac{n(n-1)}{2}$ inversions.
\end{definition}

In fact, it is known that the length of any reduced decomposition of $\sigma$ be equal to $\ell(\sigma)$. For example, $\sigma = [4~2~3~1~5]$ has five inversions, and as we saw earlier, the length of a reduced decomposition of $\sigma$ was also five.

\begin{definition}
    Let $\sigma = [\sigma_1\cdots \sigma_n] \in \mathfrak{S}_n$. A \emph{descent} occurs in position $i$ whenever $\sigma_i>\sigma_{i+1}$. We denote the set of all such indices $i$ as $\Des(\sigma)$.
    An \emph{ascent} occurs in position $j$ whenever $\sigma_j<\sigma_{j+1}$. We denote the set of all such indices $j$ as $\Asc(\sigma)$
\end{definition}

\subsection{The weak order}
The generating relations for $\mathfrak{S}_n$ can also be used to produce the lattice for the weak order poset of $\mathfrak{S}_n$.
\begin{definition}
The \JEem{right weak order poset} is a partial order defined on $\mathfrak{S}_n$, denoted $W(\mathfrak{S}_n)$. $W(\mathfrak{S}_n)$ is a bounded lattice for all $n\geq 2$. The minimal element in this lattice is the identity, while the maximal element is $w_0 = [n~n-1~\ldots 2~1]$.
\end{definition}

The cover relations are defined on the addition or removal of a single simple transposition on the right hand side. That is, $\tau \lessdot \sigma$ if and only if $\tau s_i =\sigma$ for some $i\in \Des(\sigma)$. Then $\ell (\tau) = \ell (\sigma) -1$.

In general, if $\tau \leq \sigma$, then there exists a collection of simple transpositions such that $\tau s_{i_1}\ldots s_{i_k} = \sigma$, where $\ell (\tau) = \ell(\sigma) - k$. 

For $\sigma \in \mathfrak{S}_n$, the set $\Des(\sigma)$ gives us all the elements that $\sigma$ covers in $W(\mathfrak{S}_n)$. That is, the set of all elements covered by a permutation $\sigma$ is $\{ \sigma s_i \mid i\in \Des(\sigma)\}$.


Much work has been done on subposet structures in this lattice. For example, the number of 4-cycles in $W(\mathfrak{S}_n)$ is known (OEIS \href{https://oeis.org/A317487}{A317487}). In Figure \ref{fig:S4}, we have the weak order of $\mathfrak{S}_4$, with one such 4-cycle highlighted.

\begin{figure} \label{fig:WS}
\begin{center}
\includegraphics[height=3.5in]{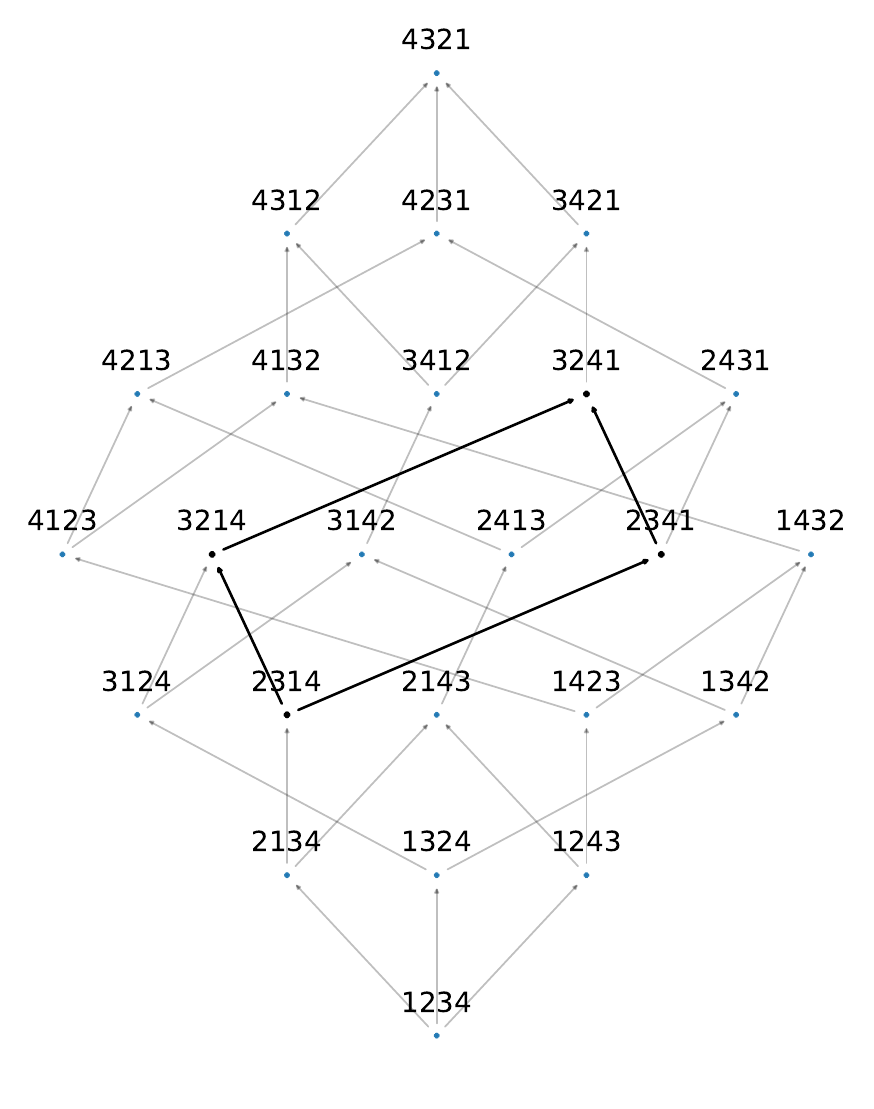}
  \caption{ The right weak order lattice for $\mathfrak{S}_4$.}
\label{fig:S4}
\end{center}
\end{figure}

Similarly, the number of 6-cycles in $W(\mathfrak{S}_n)$ is well known (OEIS \href{https://oeis.org/A317486}{A317486}). These results are unsurprising, as a 4-cycle is produced whenever $i,j\in \Des (\sigma)$ are such that $|i-j|>$, and a 6-cycle is produced whenever $i,i+1\in \Des(\sigma)$.

In \cite{Tenner_Intervals}, Tenner uses poset intervals to discuss properties of permutations and the related subposets in $W(\Sym_n)$. For example,
\begin{theorem}[\cite{Tenner_Intervals}, Corollary 4.4]
    Let $[v,w]$ be in interval in the weak order. Then $[v,w]$ is a Boolean poset if and only if $v^{-1}w$ is equal to a product of commuting generators.
\end{theorem}


\subsection{Sets of reduced words} The following recursions are well known results for sets of reduced words and braid classes. 

\begin{theorem}\label{redrec}
For $\sigma\in \mathfrak{S}_n$, 
\[
|\calR(\sigma)| = \sum_{i\in \Des(\sigma)} |\calR(\sigma s_i)|. 
\]
\end{theorem}


\begin{theorem}\label{Braidrec}
For $\sigma \in \mathfrak{S}_n$,
\[
|\calB(\sigma)| =\left ( \sum_{i\in \Des(\sigma)}|\calB(\sigma s_i)| \right )-  \sum_{i, i+1 \in \Des(\sigma)} |\calB(\sigma s_is_{i+1}s_i)|.
\]
\end{theorem}

There is a similar recursion for $|\calC(\sigma)|$ due to Elnitsky \cite{E}, but we do not use that result in this paper.

Additional work has been done by Stanley \cite{S1} which relates the size of $R(\sigma)$ to the number of standard Young tableaux of a certain size. 

\begin{theorem}[\cite{S1}, Corollary 3.1]\label{vexTab}
Let $\sigma \in \mathfrak{S}_n$. Then there exist integers $a_\lambda \geq 0$ such that 
\[
|R(\sigma )| = \sum_{\lambda \vdash \ell(\sigma)} a_\lambda f^\lambda
\]

If $\sigma$ is \emph{vexillary}, that is, 2143 pattern avoiding, there is a $\lambda \vdash \ell(\sigma)$ such that 
\[
|R(\sigma)| = f^\lambda
\]
\end{theorem}

The following result from Fishel, Milićević, Patrias, and Tenner is used repeatedly in Section \ref{sec:motivation}, as it gives a relationship between $\calR(\sigma)$ and the sets of congruence classes.

\begin{theorem}[\cite{F}, Theorem 3.6]\label{bound1}
For any permutation $\sigma$,
\[
|\calB(\sigma)|+|\calC(\sigma)|-1\leq |\calR(\sigma)|\leq |\calB(\sigma)|\cdot |\calC(\sigma)|.
\]
\end{theorem}

This theorem is what allows us to form Conjectures \ref{braidclassratio} and \ref{commclassratio}, and is used throughout the proof of Lemma \ref{casesproved} and Theorem \ref{mainresult}.

\subsection{Graphs of sets of reduced words} First, we give the formal definition of the graph $G(\sigma)$.

\begin{definition} 
Let $\sigma\in \mathfrak{S}_n$. We produce an undirected graph whose vertex set is $\calR(\sigma)$. If two words are associated by commutation moves, we give a solid edge in the graph, and if they are associated by braid moves, we give a dashed edge. We call this graph $G(\sigma)$.
\end{definition}

These edges give us straight forward methods of seeing the commutation and braid classes in the set $\calR(\sigma)$.

\begin{definition} 
Let $r_1, r_2 \in \calR(\sigma )$ for some $\sigma \in \mathfrak{S}_n$. We say that $r_1$ and $r_2$  are in the same commutation class if we can perform a series of commutation moves on $r_1$ to produce $r_2$, and vice versa. We denote the collection of commutation classes of $\sigma$ as $\calC(\sigma)$. Similarly, we say that $r_1$ and $r_2$  are in the same braid class if we can perform a series of braid moves on $r_1$ to produce $r_2$, and vice versa.  We denote the collection of braid classes of $\sigma$ as $\calB(\sigma)$.
\end{definition}

We are also interested in graphs constructed from these congruence classes, so we need some notation for those as well.

\begin{definition}\label{def:com_and_br_graphs}
    Let $\sigma\in \mathfrak{S}_n$, and consider $G(\sigma)$. 
    \begin{enumerate}
        \item The graph $G'_\calB(\sigma)$ is defined on the vertex set $V(G(\sigma))$ and the edge set, denoted $E_\calB$, is formed from the braid edges of $G(\sigma)$.
        \item The graph $G'_\calC(\sigma)$ is defined on the vertex set $V(G(\sigma))$ and the edge set, denoted $E_\calC$, is formed from the commutation edges of $G(\sigma)$.
    \end{enumerate}
\end{definition}

Note that these are not the same as the traditional congruence class graphs, where the vertices are the classes, which are formed from graph theoretic contraction and deletion procedures. 

The following result from Reiner gives a relationship between the number of braid edges in $G(\sigma)$ and the size of the set $\calR(\sigma)$.

\begin{theorem}[\cite{R}, Theorem 1]\label{ReinerBraid}
Let $d_\calB(v)$ be the number of braid edges incident to the vertex $v$ in $G(\sigma)$. For $\sigma =w_0\in \mathfrak{S}_n$,
\[
\sum_{v\in G(\sigma)}d_\calB(v) = |\calR(\sigma)|.
\]
\end{theorem}

Theorem \ref{ReinerBraid} was used by Tenner in \cite{T2} to find the average number of commutation edges incident to any vertex in $G(w_0)$. And Schilling et al. \cite{Sch} used this result to find the average number of braid edges that are incident to the connected components in $G(w_0)$ once the braid edges are deleted. We use this result to provide bounds on the number of braid edges in $G(\sigma)$ for more general classes of permutations.

\section{New notation related to equivalence of graphs}\label{sec:notation}

Throughout the paper, we look at isomorphic graphs that are produced by distinct permutations. We characterize these permutations as being \JEem{equivalent}. 

\begin{definition}
Let $\sigma, \tau \in 
\mathfrak{S}_n$ be such that $\ell(\sigma)=\ell(\tau)=l$. Let $s_{i_1}s_{i_2}\ldots s_{i_l}$ be a reduced decomposition of $\sigma$. We say that $\sigma$ and $\tau$ are \JEem{equivalent} if there exists a reduced decomposition of $\tau$, $s_{j_1}s_{j_2}\ldots s_{j_l}$ such that $i_a-i_{a+1} = j_a-j_{a+1}$ for all $1\leq a<l$. The existence of one such pair of matched decompositions means that there is a way to match both sets of reduced words with each other, and that this equivalence is on the sets of reduced words for the two permutations.
\end{definition}

For example, $\tau =[321456]$ and $\sigma = [123654]$ are equivalent permutations because we can match 5 with 2 and 4 with 1 in order to get a matching between the reduced decompositions $s_2s_1s_2$ and $s_5s_4s_5$. This also extends to the other elements in $\calR(\tau)$ and $\calR(\sigma)$. Alternately, we can see that they are equivalent permutations by viewing their graphs in Figure \ref{fig:equiv_graphs}.
\begin{figure}[H]
\begin{center}
\begin{tikzpicture}[node distance=2cm]
\node(v1)	 			            	{$121$};
\node(v2)       [right of=v1] 		    {$212$};
\node(v3)      [right of=v2]  	{$454$};
\node(v4)      [right of=v3]  	{ $545$};

\draw[dashed] (v1) --(v2);
\draw[dashed] (v3) --(v4);
\end{tikzpicture}
\vspace{0.1in}
\caption{The graphs of $G([321456])$ and $G([123654])$.}
\label{fig:equiv_graphs}
\end{center}
\end{figure}

We are particularly interested in the longest element in $\Sym_n$, and the family of permutations that are equivalent to the longest elements in some $\Sym_k$. We define this family of permutations next.

\begin{definition}\label{w0equiv}
Let $w_0^{(k,i)}$ be defined as follows
\[ 
w_0^{(k,i)}=\left ( \begin{array}{cccccccccc} 1 & \ldots & i-1 & i & i+1 & \ldots & i+k-1 & i+k & \ldots n \\
1 & \ldots &i-1 & i+k-1 & i+k-2 & \ldots & i & i+k & \ldots n
\end{array}\right )
\] 
Note that $w_0^{(k,i)}\in \mathfrak{S}_n$ is equivalent to $w_0$ in $\mathfrak{S}_k$, where $k\leq n$. We are allowed to say these are equivalent in $\mathfrak{S}_n$ because $\mathfrak{S}_k$ is a subgroup of $\mathfrak{S}_n$ for all $2\leq k \leq n$. We also note that $s_i$ is the transposition with the smallest index in the reduced decomposition for  $w_0^{(k,i)}$. This tells us that the descent set is $\{ i,i+1\ldots ,i+k-2\}$, which we sometimes refer to using the interval notation $[i,i+k-2]$.
\end{definition}

For example, $s_1s_2s_1\in \mathfrak{S}_6$ and $s_4s_5s_4\in \mathfrak{S}_6$ are both equivalent to $w_0\in \mathfrak{S}_3$, but we would write them as 
\[
w_0^{(3,1)}  = [321456],
 \mbox{ and } w_0^{(3,4)}  = [123654] .
\]

Much of our work relies on $\sigma$ and $\sigma s_i$ for $i\in \Des(\sigma)$. To conclude this section, we discuss when one of these permutations is covered by, or equivalent to, $w_0^{(k,i)}$.

\begin{proposition}\label{w0andchildren}
Let $\sigma\in \mathfrak{S}_n$ be such that $\sigma$ is not equivalent to a $w_0^{(k,i)}$, for $k>2$. Further suppose that it is not covered by a permutation equivalent to $w_0^{(k,i)}$ for any $i$. The permutations covered by $\sigma$ are $\{\sigma s_j\mid j\in \Des(\sigma) \}$.  
There exists $k,i$ such that $\sigma s_j$ is equivalent to $w_0^{(k,i)}$ for at most one $j$.
\end{proposition}

\begin{proof}
This comes from the descent sets for each $\sigma s_j$, and the lattice structure of $W(\mathfrak{S}_n)$. We consider $\sigma s_x$ and $\sigma s_y$ for $x,y\in \Des(\sigma)$. Let $\ell(\sigma )=l$. 

Suppose that $\sigma s_x = w_0^{(k_x,i)}$ and $\sigma s_y = w_0^{(k_y,j)}$. Since these permutations each have the same length, $l-1$, $k_x=k_y$ while the minimum index $i$ and $j$ could be distinct.  

Consider $\sigma s_x$ first. We can draw part of the weak order lattice containing $\sigma$ and $\sigma s_x$ in order to get an idea of what the descent sets must look like.

\begin{figure}[H]
\begin{center}
\begin{tikzpicture}[node distance=2cm]
\node(sigma)	 				{$\sigma$};
\node(sx)       [below left of=sigma] {$\sigma s_x = w_0^{(k,i)}$};
\node(sy)      [below right of=sigma]  {$\sigma s_y = w_0^{(k,j)}$};
\node(sik)      [below left of=sx]       {$w_0^{(k,i)} s_{i+k-2}$};
\node(dots1)   [left of=sik] {$\ldots$};
\node(si)      [left of=dots1]       {$w_0^{(k,i)} s_i$};

\draw(sigma)       -- (sx);
\draw(sigma)       -- (sy);
\draw(sx)       -- (si);
\draw(sx)       -- (sik);

\end{tikzpicture}
\vspace{0.1in}
\caption{Proposition \ref{w0andchildren} subposet}
\label{fig:Weakorder1}
\end{center}
\end{figure}

We write $\Des(\sigma s_x) = \{ i, i+1, \ldots , i+k-2\}$. Let $A\subset [n]$ such that $x\in A$, $A\cap [i,i+k-2] =\emptyset$ and $\{i-1,i+k-1\}\subset A$. We note that $x$ could be $i-1$ or $i+k-1$, and we could end up with $\Des(\sigma) = [i,i+k-3]\cup \{x\}$, or $\Des(\sigma) = [i+1, i+k-2] \cup \{x\}$. However, in both of these cases, $y\in \Des(\sigma)$ would be such that $\Des(\sigma s_y) \neq \Des(w_0^{(k,i)})$ for any $i,k$. 

Since $\Des(\sigma s_y)$ only contains consecutive elements, this tells us two things: (i) $x$ must be either $i-1$ or $i+k-1$, and (ii) $\Des(\sigma) = \Des(\sigma s_x) \cup \{x\} $.

Similar reasoning tells us that in order for $\Des(\sigma s_x)$ to only contain consecutive elements then $y$ must either equal $x$, or one of the following: if $x=i+k-1$, then $y=i$ and if $x=i-1$, then $y=i+k-2$. Without loss of generality, consider $\Des(\sigma ) = \{x=i-1, i, i+1, \ldots, i+(k-2)=y\}$. 

Then since we claim both $\sigma s_x$ and $\sigma s_y$ are equivalent to $w_0^{(k,j)}$ for some $j$, we have that $\sigma s_x = w_0^{(k,i)}$ and $\sigma s_y = w_0^{(k,i-1)}$. This also means we know two different reduced decompositions for $\sigma$ using decompositions for the $w_0^{(k,j)}$'s that we discussed in Chapter 1: $(s_xs_i\ldots s_{i+k-3}\ldots s_xs_is_x)s_y$ and $(s_is_{i+1}\ldots s_y\ldots s_is_{i+1} s_i)s_x$.

But this means
\[ 
\sigma = \left ( \begin{array}{cccccccccc} 1 & \ldots &i-1 & i & i+1 & \ldots & i+k-2 & i+k -1& \ldots n \\
1 & \ldots &  i+k-1 & i+k-3 & i+k-4 & \ldots &  i-1& i+k-2 & \ldots n
\end{array}\right ),
\] 
and
\[ 
\sigma =\left ( \begin{array}{cccccccccc} 1 & \ldots &i-1 & i & i+1 & \ldots & i+k-2 & i+k -1 & \ldots n \\
1 & \ldots & i&   i+k-1 & i+k-2 & \ldots & i+1 &  i-1 & \ldots n
\end{array}\right ),
\] 
which are not equivalent permutations. In fact, the only way we could have both reduced decompositions for $\sigma$ is if $k=2$, and $\sigma = s_x s_y$. Since we assume $k>2$, $\sigma$ can only cover at most one permutation equivalent to $w_0^{(k,i)}$ for some $i$.
\end{proof}


\section{Subgraphs using the Descent set of a permutation}\label{sec:subgraphs_des}
The main motivation for this section is the graph seen in Figure \ref{fig:w0_nocolor}.

\begin{figure}[H]
\begin{center}

\begin{tikzpicture}[node distance=2cm]
\node(v1)	 				{$(12132)1$};
\node(v2)       [right of=v1] 		{$(12312)1$};
\node(v3)      [below right of=v2]  	{$(12321)2$};
\node(v4)      [below right of=v3]  	{$(13231)2$};
\node(v5)      [below right of=v4]  	{$(13213)2$};
\node(v6)      [below left of=v4]  	{$(31231)2$};
\node(v7)      [below right of=v6]  	{$(31213)2$};
\node(v8)      [below left of=v7]  	{$(32123)2$};
\node(v9)      [below left of=v8]  	{$(32132)3$};
\node(v10)      [ left of=v9]  		{$(32312)3$};
\node(v11)      [below left of=v1]  	{$(21232)1$};
\node(v12)      [below left of=v11]  	{$(21323)1$};
\node(v13)      [below left of=v12]  	{$(23123)1$};
\node(v14)      [below right of=v12]  	{$(21321)3$};
\node(v15)      [below right of=v13]  	{$(23121)3$};
\node(v16)      [below right of=v15]  	{$(23212)3$};

\draw (v1) --(v2);
\draw[dashed] (v2) --(v3);
\draw[dashed] (v3) --(v4);
\draw (v4) --(v5);
\draw (v4) --(v6);
\draw (v5) --(v7);
\draw (v6) --(v7);
\draw[dashed] (v7) --(v8);
\draw[dashed] (v8) --(v9);
\draw (v9) --(v10);
\draw[dashed] (v16) --(v15);
\draw[dashed] (v10) --(v16);
\draw (v12) --(v13);
\draw (v12) --(v14);
\draw (v13) --(v15);
\draw (v14) --(v15);
\draw[dashed] (v1) --(v11);
\draw[dashed] (v11) --(v12);
\end{tikzpicture}
\vspace{0.1in}
\caption{The graph $G(w_0)$ for $w_0\in \mathfrak{S}_4$.}
\label{fig:w0_nocolor}
\end{center}
\end{figure}

This graph is large, and this is only in $\Sym_4$. In this section, we produce a family of induced subgraphs that allows us to partition $G(w_0)$ into smaller parts, which is useful when $n>4$.

\begin{definition}
An \JEem{induced subgraph} $H$ of $G$ is such that $V(H)\subset V(G)$ and $E(H)$ is defined as the subset of $E(G)$ that have both end points in $V(H)$. We denote this subgraph relation in the standard way: $H\leq G$. 
\end{definition}

\begin{proposition}\label{graph1}
Let $\sigma, \tau\in \mathfrak{S}_n$, $n\geq 2$ such that $\tau \lessdot \sigma$ in the weak order lattice. Then $G(\tau)$ is an induced subgraph of $G(\sigma)$. 
\end{proposition}

\begin{proof}
If $\tau\lessdot \sigma$ in the weak order lattice, then there is some $i\in \Des(\sigma)$ such that $\tau s_i = \sigma$. Let $t \in \calR(\tau)$. Then $ti\in \calR(\sigma)$. This is how we consider $\calR(\tau)$ as a subset of $\calR(\sigma)$. 

In fact, this is also how we can consider $G(\tau)$ as a subgraph of $G(\sigma)$. If there is a commutation edge from $t_1$ to $t_2$ in $G(\tau)$, then there is a commutation edge from $t_1i$ to $t_2i$ in $G(\sigma)$. The same applies for braid edges in $G(\tau)$. Thus we can map edges from $G(\tau)$ into $G(\sigma)$. 

Additionally, we can look at $V(G(\tau))$ as a subset of $V(G(\sigma))$, and consider an edge in $G(\sigma)$ incident to two vertices contained in $V(G(\tau))$. This edge must already be contained in $G(\tau)$, because any commutation or braid move is happening between words in $\calR(\tau)$, since these vertices are of the form $\tau_1s_i$ and $\tau_2 s_i$.

Therefore, $G(\tau)$ is an induced subgraph of $G(\sigma)$.
\end{proof}

\begin{corollary}
Let $\sigma, \tau\in \mathfrak{S}_n$, $n\geq 2$ such that $\tau \leq \sigma$ in the weak order lattice. Then,
\begin{enumerate}
    \item $G(\tau)$ is an induced subgraph of $G(\sigma)$,
    \item $G(\tau)$ is an induced subgraph of $G(w_0)$.
\end{enumerate}
\end{corollary}

\begin{proof}
Consider the chain in the weak order lattice:
\[
\tau \lessdot \tau_1 \lessdot \tau_2 \lessdot \ldots \lessdot \tau_k \lessdot \sigma,
\]
where each element in the chain covers the previous element. We know such a chain exists because $\tau \leq \sigma$. Note that every element in $W(\Sym_n)$ is comparable to $w_0$, so for any $\tau$, we could set $\sigma = w_0$.

From the previous proposition, we can construct the following chain of induced subgraphs:
\[
G(\tau) \leq G(\tau_1) \leq G(\tau_2) \leq \ldots \leq G(\tau_k) \leq G(\sigma).
\]
So we have that $G(\tau) \leq G(\sigma)$, as desired.
\end{proof}

We wish to combine our previous results into a method of partitioning the vertex set of $G(\sigma)$.

\begin{theorem}\label{graph_partition}
     Let $\sigma \in \Sym_n$. Then the set $V(G(\sigma))$ may be partitioned into disjoint subsets of the form $V(G(\sigma s_i))$ for $i\in \Des(\sigma s_i)$.
\end{theorem}

\begin{proof}
    As noted in Theorem \ref{redrec}, the set $\calR(\sigma)$ can be partitioned into $\calR(\sigma s_i)$ for $i\in \Des(\sigma)$. We simply apply the labeling of the subgraphs $G(\sigma s_i)$ to the vertex set of $G(\sigma)$ as in Proposition \ref{graph1}, and we have our partition of the vertex set.
\end{proof}

Using this partition of subgraphs, we wish to study the degree of vertices in $G(\sigma)$. 

\begin{definition}
Let $\sigma \in \mathfrak{S}_n$. For $G(\sigma)$, we define $d_\calB(v)$ to be the braid edge degree of a vertex $v$. That is, $d_\calB(v)$ counts the number of braid edges incident to $v$ in the graph $G(\sigma)$.

Similarly, we define $d_\calC(v)$ to be the commutation edge degree of a vertex $v$. That is, $d_\calC(v)$ counts the number of commutation edges incident to $v$ in the graph $G(\sigma)$.
\end{definition}

\begin{example} 
Consider again $w_0\in \mathfrak{S}_4$. The graph $G(w_0)$ has been drawn in Figure \ref{fig:w0}, and the vertices of $G(w_0)$ have been partitioned.

\begin{figure}[H]
\begin{center}

\begin{tikzpicture}[node distance=2cm]
\node(v1)	 				{$(12132)1$};
\node(v2)       [right of=v1] 		{$(12312)1$};
\node(v3)      [below right of=v2]  	{\color{blue}{$(12321)2$}};
\node(v4)      [below right of=v3]  	{\color{blue}{$(13231)2$}};
\node(v5)      [below right of=v4]  	{\color{blue}{$(13213)2$}};
\node(v6)      [below left of=v4]  	{\color{blue}{$(31231)2$}};
\node(v7)      [below right of=v6]  	{\color{blue}{$(31213)2$}};
\node(v8)      [below left of=v7]  	{\color{blue}{$(32123)2$}};
\node(v9)      [below left of=v8]  	{\color{red}{$(32132)3$}};
\node(v10)      [ left of=v9]  		{\color{red}{$(32312)3$}};
\node(v11)      [below left of=v1]  	{$(21232)1$};
\node(v12)      [below left of=v11]  	{$(21323)1$};
\node(v13)      [below left of=v12]  	{$(23123)1$};
\node(v14)      [below right of=v12]  	{\color{red}{$(21321)3$}};
\node(v15)      [below right of=v13]  	{\color{red}{$(23121)3$}};
\node(v16)      [below right of=v15]  	{\color{red}{$(23212)3$}};

\draw (v1) --(v2);
\draw[dashed,green] (v2) --(v3);
\draw[dashed,blue] (v3) --(v4);
\draw[blue](v4) --(v5);
\draw[blue] (v4) --(v6);
\draw[blue] (v5) --(v7);
\draw[blue] (v6) --(v7);
\draw[dashed,blue] (v7) --(v8);
\draw[dashed,green] (v8) --(v9);
\draw[red] (v9) --(v10);
\draw[dashed,red] (v16) --(v15);
\draw[dashed,red] (v10) --(v16);
\draw (v12) --(v13);
\draw[green] (v12) --(v14);
\draw[green] (v13) --(v15);
\draw[red] (v14) --(v15);
\draw[dashed] (v1) --(v11);
\draw[dashed] (v11) --(v12);
\end{tikzpicture}
\vspace{0.1in}
\caption{The graph $G(w_0)$ for $w_0\in \mathfrak{S}_4$, with induced subgraphs highlighted.}
\label{fig:w0}
\end{center}
\end{figure}

The induced subgraphs highlighted in different colors: $G(w_0s_1)$ is drawn black with the 1 in the right most position in the reduced words, $G(w_0s_2)$ in blue with a 2 in the right most position, and $G(w_0s_3)$ in red with a three in the right most position. There are also four edges in green that only appear in $G(w_0)$ connecting those subgraphs.
\end{example}

The edges between subgraphs are what we are particularly interested in. Note that the braid edge degree of the vertex $\overline{u} = (32132)3$ is different depending on whether we view it as $u\in V(G([4312]))$ or $\overline{u} = u3\in V(G(w_0))$. From this observation, we can now produce results for the braid edge degree of a vertex in $G(\sigma)$.

\begin{proposition}\label{newEDGE}
Let $\sigma \in \mathfrak{S}_n$, and $i,j\in \Des(\sigma)$. Suppose that $ws_i, us_j \in \calR(\sigma)$, where $w\in \calR(\alpha)$, $u\in \calR(\beta)$ and $\alpha, \beta \lessdot \sigma$. Suppose that $(ws_i)- (us_j)$ is an edge in the graph $G(\sigma)$. 
\begin{enumerate}
\item If $u\neq w$, but $i=j$, then $\alpha = \beta$ and the edge appears in the induced subgraph $G(\alpha)$. 
\item If $i\neq j$, then the vertices appear in two disjoint subgraphs of $G(\sigma)$, and so this edge is not properly contained in any induced subgraph.
\end{enumerate}
\end{proposition}

\begin{proof}
Case 1 was already proven in Proposition \ref{graph1}. 
For Case 2, if $i\neq j$, then $ws_i$ appears in $\calR(\sigma s_i)$, which is disjoint from $\calR(\sigma s_j)$ which contains $us_j$. This means that the two vertices appear in disjoint induced subgraphs. Thus the edge $(ws_i)- (us_j)$ connects the two subgraphs in $G(\sigma)$, but is not properly contained in either one. 
\end{proof}

Now that we have a better idea of what edges appear only in $G(\sigma)$ and not in any of the induced subgraphs, we have the following proposition:

\begin{proposition}\label{braidrec11}
Let $\sigma, \tau \in \mathfrak{S}_n$, $n\geq 2$ such that $\tau \lessdot \sigma $ in the weak order lattice. Suppose that $\tau s_i = \sigma$, where $i\in \Des(\sigma)$.
\begin{enumerate}
\item If $i-1,~i,~i+1\in \Des(\sigma)$, then there are $|\calR(\sigma s_is_{i+1}s_i)|+|\calR(\sigma s_is_{i-1}s_i)|$ braid edges incident to the vertices of $G(\tau)$ that are not contained in the induced subgraph $G(\tau)$. 
\item If only $i,~i+1\in \Des(\sigma)$, (or $i-1,i\in \Des(\sigma)$), then there are $|\calR(\sigma s_is_{i+1}s_i)|$ braid edges incident to the vertices of $G(\tau)$ that are not contained in the induced subgraph $G(\tau)$. 
\item If neither of these cases is true, then the only braid edges incident to the vertices of $G(\tau)$ are the braid edges that are incident to two vertices from $G(\tau)$.  
\end{enumerate}
\end{proposition}

\begin{proof}
Consider $\sigma, \tau \in \mathfrak{S}_n$ such that $\tau \lessdot \sigma $, and suppose that $ i, i+1\in \Des(\sigma)$ and $\sigma s_i=\tau $. 

Note that since $i,i+1\in \Des(\sigma)$, then $\sigma = [a_1\ldots a_n]$ is such that $a_i>a_{i+1}>a_{i+2}$. We can consider $\sigma s_i= [\ldots  a_{i+1}a_ia_{i+2}\ldots] $, and $\sigma s_{i+1} = [\ldots  a_ia_{i+2}a_{i+1}\ldots]$. We note that $i\in \Des(\sigma s_{i+1})$ and $i+1\in \Des(\sigma s_i)$. We can continue in this manner to get the portion of weak order lattice in Figure \ref{fig:i+1lattice}.

\begin{figure}[H]
\begin{center}
\begin{tikzpicture}[node distance=1.2cm]
\node(sigma)	 	{$\sigma$};
\node(sx)       [below left of=sigma] {$\sigma s_i$};
\node(sy)      [below right of=sigma]  {$ \sigma s_{i+1}$};
\node(sik)      [below of=sx]       {$\sigma s_is_{i+1} $};
\node(dots1)   [below right of =sik] {$\sigma s_is_{i+1}s_i $};
\node(si)      [below of=sy]       {$\sigma s_{i+1} s_i$};

\draw(sigma)       -- (sx);
\draw(sigma)       -- (sy);
\draw(sx)       -- (sik);
\draw(sy)       -- (si);
\draw(si) --(dots1);
\draw(sik) --(dots1);

\end{tikzpicture}
\vspace{0.1in}
\caption{When $i,i+1\in \Des(\sigma)$, we have this subposet in $W(\mathfrak{S}_n)$}
\label{fig:i+1lattice}
\end{center}
\end{figure}

Let $\alpha= \sigma s_is_{i+1}s_i$ so that $\sigma = \alpha s_is_{i+1}s_i$. Equivalently, we have that $\sigma=\alpha s_{i+1}s_is_{i+1}$. 

We want to be able to count the number of reduced words of $\sigma$ that have a braid move in those last three positions that use the letters $i$ and $i+1$. This means we really need to count the number of ways we can write $\alpha$. And in fact, there are $|\calR(\sigma s_is_{i+1}s_i)|$ ways to write $\alpha$.

We note that all vertices that end in the letters $i,~i+1,~i$ are contained in the subgraph $G(\sigma s_i)$, where $\tau = \sigma s_i$. All vertices that end in the letters $i+1,~i,~i+1$ are contained in the subgraph $G(\sigma s_{i+1})$. Then the number of braid edges between these subgraphs are equal to $|\calR(\sigma s_is_{i+1}s_i)|$.

The same reasoning allows us to conclude that if $i-1,~i,~i+1\in \Des(\sigma)$, there are $|\calR(\sigma s_is_{i+1}s_i)|+|\calR(\sigma s_is_{i-1}s_i)|$ ``new" braid edges incident to vertices in $G(\tau)$. 

Similarly, if we cannot write $\sigma =\alpha s_i s_{i\pm 1}s_i$ for this particular $i$, then the only possible braid edges that can be incident to the vertices of $G(\tau)$ must be fully contained in the induced subgraph.
\end{proof}

The following recursive formula is an immediate consequence of Proposition \ref{braidrec11}.

\begin{corollary}\label{recDEG}
Let $\sigma \in \mathfrak{S}_n$. Then
\[
\sum_{v\in G(\sigma)} d_\calB(v) =\left ( \sum_{i\in \Des(\sigma)}\sum_{u\in G(\sigma s_i)}d_\calB(u) \right )+ 2 \cdot \sum_{i,i+1\in \Des(\sigma)}|\calR(\sigma s_is_{i+1}s_i)|,
\]
where we consider $d_\calB(u)$ in $G(\sigma s_i)$ on the right hand side, and $\overline{u}=ui \in G(\sigma)$ on the left hand side.
\end{corollary}

This is a brand new way to count the number of braid edges in a graph of a set of reduced words. Now we have methods of breaking $G(\sigma)$ into subgraphs, and a way to count the braid edges in $G(\sigma)$. Since we want to understand the congruence classes in $\calB(\sigma)$, this is a step closer to that goal.

Because the arguments are very similar, we also prove a result that allows us to count the number of commutation edges in $G(\sigma)$. 

\begin{proposition}\label{braidedge10}
Let $\sigma, \tau \in \mathfrak{S}_n$, $n\geq 2$ such that $\tau \lessdot \sigma $ in the weak order lattice. Suppose that $\tau s_i = \sigma$, where $i\in \Des(\sigma)$. Let $I_i = \{j\mid j\in \Des(\sigma) \mbox{ and } |j-i|>1 \} $.
\begin{enumerate}
\item If $I_i \neq \emptyset$,  then there are $\sum_{j\in I_i}|\calR(\sigma s_is_j)|$ commutation edges incident to the vertices of $G(\tau)$ that are not contained in the induced subgraph. 
\item If $I_i = \emptyset$, then the only commutation edges incident to the vertices of $G(\tau)$ are the commutation edges that are incident to two vertices from $G(\tau)$.  
\end{enumerate}
\end{proposition}

\begin{proof}
Consider $\sigma, \tau \in \mathfrak{S}_n$ such that $\tau \lessdot \sigma $. Suppose that $ i, j\in \Des(\sigma)$ are such that $|i-j|>1$, and $\sigma s_i=\tau $. 

Since $ i, j\in \Des(\sigma)$, and letters $i$ and $j$ commute with each other, there exist reduced decompositions of $\sigma$ such that $\sigma = \alpha s_is_j $ and $\sigma=\alpha s_js_i$. We want to be able to count the number of reduced words of $\sigma$ that have a commutation move in those last two positions that use the letters $i$ and $j$. This means we really need to count the number of ways we can write $\alpha$. There are $|\calR(\sigma s_is_j)|$ ways to write $\alpha$.

We note that all vertices that end in the letters $j~i$ are contained in the subgraph $G(\sigma s_i)$, where $\tau = \sigma s_i$. All vertices that end in the letters $i~j$ are contained in the subgraph $G(\sigma s_{j})$. Then the number of commutation edges between these subgraphs are equal to $|\calR(\sigma s_is_j)|$.

The same reasoning allows us to conclude that for all $k\in I_i$, then there are $|\calR(\sigma s_k s_i)|$ ``new" commutation edges incident to vertices in $G(\tau)$. 

Similarly, if  $I_i=\emptyset$ then we cannot write $\sigma =\beta s_j s_i = \beta s_i s_j$ for this particular $i$, then the only possible commutation edges that can be incident to $G(\tau)$ must be fully contained in the induced subgraph.
\end{proof}

The following is an immediate consequence of Proposition \ref{braidedge10}.

\begin{corollary}\label{coro_edge2}
Let $\sigma \in \mathfrak{S}_n$. Then if $d_\calC(v)$ counts the number of commutation edges incident to $v$ in $G(\sigma)$,
\[
\sum_{v\in G(\sigma)} d_\calC(v) =\left ( \sum_{i\in \Des(\sigma)}\sum_{u\in G(\sigma s_i)}d_\calC(u) \right )+ 2 \cdot \sum_{i,j\in \Des(\sigma) ~i-j>1}|\calR(\sigma s_j s_i)|,
\]
where we consider $d_\calB(u)$ in $G(\sigma s_i)$ on the right hand side, and $\overline{u}=ui \in G(\sigma)$ on the left hand side.
\end{corollary}

Note that the subgraphs in this section have been, in many ways, the most natural method of partitioning our graphs into subgraphs. The related edge degree recursions are certainly useful, but as we see in Section \ref{sec:motivation}, they only get us so far when we consider our main motivating problem. 

However, as the results in this section are brand new tools, we are interested in what applications they can have to other problems related to $\calR(\sigma)$ and $W(\Sym_n)$.

\section{Writing reduced words based on the descent set}\label{sec:shuffle_sub}

As we worked on various examples proofs in the previous section, we began asking questions about how we were generating the set $\calR(\sigma)$. If we had a single reduced word for $\sigma$, what shortcuts we were taking to generate these sets quickly, and what other types of sub-structures did we notice?

\begin{example}\label{blockededges}
Consider the permutation $\sigma = [2431]=s_1s_2s_3s_2$ with descent set $\Des(\sigma) = \{ 2,3\}$.
\begin{figure}[H]
\begin{center}
\begin{tikzpicture}[node distance=2cm]
\node(v1)	 			            	{$1232$};
\node(v2)       [above of=v1] 		    {$1323$};
\node(v3)      [right of =v2]  	{$3123$};

\draw[dashed] (v1) --(v2);
\draw (v2) --(v3);
\end{tikzpicture}
\vspace{0.1in}
\caption{The graph $G([2431])$}
\label{fig:C4E1}
\end{center}
\end{figure}

We note that the braid edge itself is a copy of $G([1432])$, which is $w_0^{(3,2)}$, and has the same descent set as $\sigma$. Starting with the words 1323 and 1232, we attempt to move the letter 1 to the right. We note that after moving the letter 1, we do not have a second full copy of $G([1432])$. We would need both the vertices 2132 and 3123, as well as a braid edge between them. There are two problems with this: there is no braid move between those reduced words, and 2132 is not a reduced word of $\sigma$.

Next we consider $\sigma = [32154]=s_4s_1s_2s_1$ with descent set $\Des(\sigma) = \{1,2,4 \}$. We have 

\begin{figure}[H]
\begin{center}
\begin{tikzpicture}[node distance=2cm]
\node(v1)	 			            	{$4121$};
\node(v2)       [above of=v1] 		    {$4212$};
\node(v3)      [right of=v2]        	{$2412$};
\node(v4)      [right of=v1]        	{$1421$};
\node(v5)      [right of=v3]        	{$2142$};
\node(v6)      [right of=v4]  	        {$1241$};
\node(v7)      [right of=v5]        	{$2124$};
\node(v8)      [right of=v6]  	        {$1214$};

\draw[dashed] (v1) --(v2);
\draw (v2) --(v3);
\draw (v1) --(v4);
\draw (v5) --(v3);
\draw (v6) --(v4);
\draw (v5) --(v7);
\draw (v6) --(v8);
\draw[dashed] (v7) --(v8);
\end{tikzpicture}
\vspace{0.1in}
\caption{The graph $G([32154])$}
\label{fig:C4E0}
\end{center}
\end{figure}

In this picture, we have four copies of the vertex set of $G([321])$, with 4 sitting in different spots through the reduced words for $121$. We focus on this permutation because the longest string of consecutive elements in $\Des(\sigma)$ is the descent set for $w_0^{(3,1)}$. As the letter 4 moves through the word from left to right, we have effectively blocked braid edges from appearing in the middle portion of the graph. 
\end{example}

The example in Figure \ref{fig:C4E0}, and notation found in  \cite{Tenner_Intervals} leads us to the following idea: rather than focus on intervals $[\sigma s_i, \sigma]$ in the weak order lattice, we can focus on $[\tau, \sigma]$ such that $\tau^{-1}\sigma = w_0^{(k,i)}$, for some $i,k\in \N$.  

\begin{lemma}\label{breakingPerms}
Let $\sigma \in \mathfrak{S}_n$. Suppose that $\Des(\sigma)$ contains a string of $m$ consecutive elements, with smallest element in the string $i$: $\{i,i+1,\ldots, i+m-1\}$. Then there exists a $u\in \Sym_n$ such that $u\leq \sigma$, and $u^{-1}\sigma = w_0^{(k,i)}$, for $k=m+1$.
\end{lemma}

\begin{proof}
Let $\sigma = [a_1 a_2\ldots a_n ] \in \mathfrak{S}_n$. We recall that $i\in \Des(\sigma)$ if and only if $a_i>a_{i+1}$. We also recall that for any $j$, $\sigma s_j = [a_1 a_2\ldots a_{j+1}a_j \ldots a_n]$. 

For any element $i\in \Des(\sigma)$, we recall that $s_i = w_0^{1,i}$. Suppose that $\{i,i+1,\ldots i+m-1 \} \subset \Des(\sigma)$ is a string of consecutive elements where $m\geq 2$. This means that 
\[
a_i>a_{i+1}> \ldots>a_{i+m-1}.
\]

Let $\tau_{\sigma}^{(j,k)}$ be defined as follows: 
\[
\tau_{\sigma}^{(j,k)} := \sigma s_js_{j+1}\ldots s_{j+k-1} = [a_1 \ldots a_{j-1}a_{j+1}\ldots a_{j+k}a_j\ldots a_n],
\]
where $\tau_{\sigma}^{(0,0)}=\sigma$, and  $\tau_{\sigma}^{(j,1)} = \sigma s_j $.

We note that for the set $\{i,i+1,\ldots i+m-1 \} \subset \Des(\sigma)$, if $j=i$ and $k=m$, we have
\[
\tau_{\sigma}^{(i,m)}=\sigma s_is_{i+1}\ldots s_{i+m-1} = [a_1\ldots a_{i-1}a_{i+1}a_{i+2}\ldots a_{i+m}a_i\ldots a_n].
\]
We see that $\{i,i+1,\ldots i+m-2 \} \subset \Des(\tau_{\sigma}^{(i,m)})$, while $i+m-1\notin \Des(\tau_{\sigma}^{(i,m)})$. 

Furthermore, $\ell(\tau_{\sigma}^{(i,m)})= \ell(\sigma) -m$. Since $i\in \Des(\sigma)$, $\ell(\tau_{\sigma}^{(i,1)})=\ell(\sigma)-1$. We ~ still have $i+1\in \Des(\tau_{\sigma}^{(i,1)})$, and we can see that $\ell(\tau_{\sigma}^{(i,2)}) = \ell(\sigma)-2$. Inductively, this process continues until we have $\ell(\tau_{\sigma}^{(i,m)}) = \ell(\sigma)-m$. 

Let $\sigma^{(1)} =\tau_{\sigma}^{(i,m)}$. We now consider 
\[
\tau_{\sigma^{(1)}}^{(i,m-1)}= \sigma^{(1)}s_is_{i+1}\ldots s_{i+m-2}  = [a_1\ldots a_{i-1}a_{i+2}\ldots a_{i+m}a_{i+1}a_i\ldots a_n].
\]
We note that 
\[i+m-1, i+m-2 \notin \Des(\tau_{\sigma^{(1)}}^{(i,m-1)}) \mbox{ and }\{i,i+1,\ldots i+m-3 \} \subset \Des(\tau_{\sigma^{(1)}}^{(i,m-1)}).
\]
We also have $\ell(\tau_{\sigma^{(1)}}^{(i,m-1)}) = \ell(\sigma^{(1)}) - (m-1) = \ell(\sigma) - m -(m-1)$ using the same argument on descents as before. Let $\sigma^{(2)} = \tau_{\sigma^{(1)}}^{(i,m-1)}$.

Inductively, for $1\leq x\leq m-3$, we have that the permutation $\sigma^{(x+1)} = \tau_{\sigma^{(x)}}^{(i,m-x)}$ is such that
\[
\sigma^{(x+1)} =  \sigma^{(x)}s_is_{i+1}\ldots s_{i+m-(x+1)}= [a_1\ldots a_{i-1}a_{i+x}\ldots a_{i+m}a_{i+x-1}\ldots a_{i+1}a_i\ldots a_n],
\]
where $\{i,\ldots, i+m-(x+2) \}  \subset \Des(\sigma^{(x+1)} )$, $\{i+m-(x+1), \ldots , i+m-1 \} \cap \Des(\sigma) = \emptyset$, and $\ell(\sigma^{(x+1)} ) = \ell(\sigma) - \sum_{b=0}^{m-x}m-b$. 

Now consider $\sigma^{(m-2)} = \tau_{\sigma^{(m-3)}}^{(i,m-(m-3))}$:
\[
\sigma^{(m-2)} = [a_1\ldots a_{i-1}a_{i+m-2}a_{i+m-1}a_{i+m}a_{i+m-3}\ldots a_{i+1}a_i\ldots a_n].
\]
From the above one line notation, we see that $\sigma^{(m)}:=\sigma^{(m-2)} s_i s_{i+1}s_i$ such that $\{i,i+1,\ldots i+m-2 \} \cap \Des(\sigma^{(m)}) = \emptyset$. Additionally, 
\[
\ell(\sigma^{(m)}) =\left ( \ell(\sigma) - \sum_{b=0}^{m-3}m-b\right ) -3 =  \ell(\sigma) - \binom{m+1}{2}.
\]

Let $u\in \calR(\sigma^{(m)})$. Note that if $\sigma^{(m)}$ is the identity, then $u$ is simply the empty word. In either case, by tracing our products from $\sigma^{(m)}$ back to $\sigma^{(1)}$, let $v$ be the word
\[
i(i+1)i~(i+2)(i+1)i\ldots (i+k-1)(i+k-2)\ldots (i+1)i.
\]
We see that $\ell(v) = \binom{m+1}{2}$, and that we can verify that $v\in \calR(w_0^{(m+1,i)})$. 
By construction, we have $uv\in \calR(\sigma)$ as desired.
\end{proof}

This result leads us to the following, more general theorem.

\begin{theorem}\label{thm:descentsetbreak}
    Let $\sigma\in \mathfrak{S}_n$. Suppose that $\Des(\sigma)$ can be partitioned into maximal blocks of consecutive entries, $B_1, B_2,\ldots B_m$. Let $\min(B_j) = b_j$ for all $1\leq j\leq m$. Then
    \[
    \sigma = \tau \prod_{i=1}^m w_0^{(|b_j|+1,b_j)}
    \]
    for some $\tau \leq \sigma$.
\end{theorem}

\begin{proof}
    For each block of consecutive descents, we apply Lemma \ref{breakingPerms} to produce \[\sigma = \tau_1 w_0^{(|b_1|+1,b_1)}.\]
    We apply  Lemma \ref{breakingPerms} again to arrive at 
    \[
    \sigma = \tau_2 w_0^{(|b_2|+1,b_2)}w_0^{(|b_1|+1,b_1)}.\] 
    We note that as the descent sets are disjoint, then
    \[
    w_0^{(|b_2|+1,b_2)}w_0^{(|b_1|+1,b_1)} = w_0^{(|b_1|+1,b_1)}w_0^{(|b_2|+1,b_2)}
    \]
    so the order does not matter.
    
    We repeat this process inductively until we have worked our way down the lattice to $\tau$, and the product 
     \[
    \sigma = \tau \prod_{i=1}^m w_0^{(|b_j|+1,b_j)}
    \]
    as desired.
\end{proof}

Using Theorem \ref{thm:descentsetbreak}, we have the following classification of the shortest permutation with a fixed descent set.

\begin{corollary}\label{cor:minimalperm}
    Let $\sigma\in \mathfrak{S}_n$. Suppose that $\Des(\sigma)$ can be partitioned into maximal blocks of consecutive entries, $B_1, B_2,\ldots B_m$. Then
    \[
    \sigma' = \prod_{i=1}^m w_0^{(|b_j|+1,b_j)}
    \]
    is the permutation with minimal length such that $\Des(\sigma') = \Des(\sigma)$.
\end{corollary}

\begin{proof}
    From Theorem \ref{thm:descentsetbreak}, we note that for $\Des(\sigma)$, we must have 
    \[
    \prod_{i=1}^m w_0^{(|b_j|+1,b_j)}
    \]
    as part of a reduced decomposition of $\sigma$. Then $\tau = e$ produces the permutation with minimal length with this descent set.
\end{proof}

\section{Shuffles Subgraphs}\label{subsec:shuffles}

Using ideas from the previous section, we now discuss how to consider subgraphs in $G(\sigma )$ that may not be of the form $G(\sigma s_i)$, $i\in \Des(\sigma )$.

\begin{definition}\label{shuffledef}
Let $\sigma \in \mathfrak{S}_n$ be such that $[i,i+k-2]\subset \Des(\sigma)$ is a set of consecutive elements. From Lemma \ref{breakingPerms}, we note that there are permutations $\alpha, \beta \in \mathfrak{S}_n$ such that $\beta =w_0^{(k,i)}$, and $u=u_1u_2\ldots u_{l(\alpha)}\in \calR(\alpha)$ and $v=v_1v_2\ldots v_{l(\beta)}\in \calR(w_0^{(k,i)})$ can be concatenated to produce $uv\in \calR(\sigma)$. That is, $\sigma = \alpha \beta$. In order to distinguish the letters, we color the letters of $u$ blue, and the letters of $v$ red.

A \JEem{shuffle} of the letters of $uv$ is defined in the following way:
\begin{enumerate}
    \item A commutation or braid move using only the blue letters of $u$, or using only the red letters of $v$.
    \item A commutation move using one blue letter $u_i$ of $u$ and one red letter $v_j$ of $v$.
    \item A braid move that uses two blue letters $u_i u_{i+1}$ of $u$ and one red letter $v_j$ of $v$, or vice versa.
\end{enumerate}
\end{definition}

None of the above shuffle types change the color of the letters. After any finite sequence of shuffles of any type, we no longer have the word $uv$. We still refer to each new shuffle as a shuffle of the letters of $uv$. 

\begin{example}
Consider $\sigma = [165324]=s_4s_5 s_2s_3s_4s_2 s_3 s_2$. We can use Lemma \ref{breakingPerms} to write $45\in \calR(\alpha)$, $234232\in \calR(\beta)$, and ${\color{blue}45}{\color{red}234232}\in \calR(\sigma)$.
\begin{enumerate}
    \item Using a shuffle of type 1, we have ${\color{blue}45}{\color{red}434234}\in \calR(\sigma)$.
    \item Using a shuffle of type 2, we have ${\color{blue}4}{\color{red}2}{\color{blue}5}{\color{red}34232}\in \calR(\sigma)$.
    \item Using a shuffle of type 3, we have ${\color{blue}54}{\color{red}534234}\in \calR(\sigma)$
\end{enumerate}
\end{example}

We see that for the first type of shuffle, we start with $u\in \calR(\alpha)$ and shuffle letters to get $a\in \calR(\alpha)$. For a fixed pair $a\in \calR(\alpha)$ and $b\in \calR(\beta)$, the second type of shuffle commutes where the letters of $a$ and $b$ sit. The third shuffle type is a mix of blue letters forming a word $a'$, and red letters form a word $b'$. However, $a'\notin \calR(\alpha)$ and $b'\in \calR(\beta)$, which makes the words formed after a shuffle of type 3 more difficult to discuss.

We know that we can start with any word in $\calR(\sigma)$, and generate the full set by performing all possible commutation and braid moves among the letters. So we can use the letters of $uv$ to fully generate $\calR(\sigma)$ in the standard way. 

The question is whether we can keep track of the letters of $u\in \calR(\alpha)$ and $v\in \calR(\beta)$ as we perform the shuffle process?

\begin{lemma}\label{permsSuffle}
Let $\alpha, \beta, \sigma\in \mathfrak{S}_n$, with $u\in \calR(\alpha)$, $v\in \calR(\beta)$ and $uv\in \calR(\sigma)$ as described in Lemma \ref{breakingPerms} and Definition \ref{shuffledef}. We can fully construct $\calR(\sigma)$ by looking at all possible ways that the letters of $uv$ can shuffle through each other. 
\end{lemma}

\begin{proof}
Consider $uv\in \calR(\sigma)$, where $u = u_1\ldots u_{l(\alpha)}\in \calR(\alpha)$ and $v = v_1\ldots v_{l(\beta)}\in \calR(\beta)$. We color all the letters descended from $u$ blue, and all the letters descended from $v$ red. 

Let $w\in \calR(\sigma)$ be an arbitrary word that is distinct from $uv$. We know that there is a finite sequence of commutation and braid moves that transform $uv$ into $w$. Since $G(\sigma)$ is connected, let us consider this as a path of $m$ vertices in $G(\sigma)$.
\[
(uv) - (a_2) - \cdots -(a_{m-1}) - (w)
\]

To travel from $uv$ to $a_2$, we either perform a commutation move or a braid move of the letters of $uv$. This is a shuffle of the letters of $uv$. 

Inductively, each $a_j$ in this path has all the letters colored blue and red, since no shuffle type changes the colors of the letters. We always have the same number of blue letters and red letters, since none of the shuffles recolors the letters. Thus, the letters of $w$ is a mix of $\ell(\alpha)$ blue letters and $\ell(\beta)$ red letters.

Because $w$ was an arbitrary element of $\calR(\sigma)$, every element of the set is formed from a series of shuffles of the letters of $uv$.
\end{proof}

Our work on splitting $\sigma$ into pieces $\alpha$ and $\beta$, and shuffling the letters of their respective reduced words, is similar to permutation inflations. Permutation inflations are related to grid drawings of permutations, and use patterns in the one line notation in consecutive spots to write reduced words. For more information on inflations, we recommend \cite{Alb} or \cite{SDYK}. 

We pursued this new shuffle method rather than the inflations because we could use our $\alpha$ and $\beta$ construction for any arbitrary permutation, rather than being restricted to a particular family. 

\begin{example}
With the knowledge from Lemma \ref{permsSuffle}, we look at a slightly more complex example of $G(\sigma)$, though we do not draw the full graph: $\sigma = [3215476]=s_4s_6 s_1 s_2 s_1$ with descent set $\Des(\sigma) = \{1,2,4,6 \}$. The longest string of consecutive descents is $\{1,2\}$. 

From Lemma \ref{breakingPerms} $\alpha = [1235476]$ and $\beta = [3214567]$. Note that as we shuffle $a\in \calR(\alpha)$ and $b\in \calR(\beta)$, we only perform shuffle moves of the first and second type. We consider ${\color{blue}46}{\color{red}121}\in \calR(\sigma)$ as our starting point.

For any reduced word of $\sigma$, we can select where $121$ or $212$ sits, and then the remaining two spots can have either $4$ or $6$. Thus we have $|\calR(\sigma)| = 2\binom{5}{3} \cdot 2 =40$.

We can begin to construct the graph as follows:
\begin{figure}[H]
\begin{center}
\begin{tikzpicture}[node distance=2cm]
\node(v1)	 			            	{${\color{blue}46}{\color{red}121}$};
\node(v2)       [above of=v1] 		    {${\color{blue}46}{\color{red}212}$};
\node(v3)      [right of=v2]        	{${\color{blue}64}{\color{red}212}$};
\node(v4)      [right of=v1]        	{${\color{blue}64}{\color{red}121}$};
\node(v5)      [below right of=v3]        	{${\color{blue}4}{\color{red}2}{\color{blue}6}{\color{red}12}$};
\node(v6)      [below right of=v4]  	        {${\color{blue}4}{\color{red}1}{\color{blue}6}{\color{red}21}$};
\node(v7)      [right of=v5]        	{${\color{blue}6}{\color{red}2}{\color{blue}4}{\color{red}12}$};
\node(v8)      [right of=v6]  	        {${\color{blue}6}{\color{red}1}{\color{blue}4}{\color{red}21}$};

\node(v9)      [below right of=v7]        	{${\color{blue}4}{\color{red}21}{\color{blue}6}{\color{red}2}$};
\node(v10)      [below right of=v8]  	        {${\color{blue}4}{\color{red}12}{\color{blue}6}{\color{red}1}$};
\node(v11)      [right of=v9]        	{${\color{blue}6}{\color{red}21}{\color{blue}4}{\color{red}2}$};
\node(v12)      [right of=v10]  	        {${\color{blue}6}{\color{red}12}{\color{blue}4}{\color{red}1}$};

\draw[dashed] (v1) --(v2);
\draw (v2) --(v3);
\draw (v1) --(v4);
\draw [dashed] (v3) --(v4);
\draw (v5) --(v2);
\draw (v6) --(v1);
\draw (v7) --(v3);
\draw (v8) --(v4);
\draw (v5) --(v9);
\draw (v6) --(v10);
\draw (v7) --(v11);
\draw (v8) --(v12);
\end{tikzpicture}
\vspace{0.1in}
\caption{Part of the graph $G([3215476])$}
\label{fig:C4E2}
\end{center}
\end{figure}

The square of vertices to the far left can be viewed as a copy of the standard product graph $G([1235476]) \times G([3214567])$. The commutation edges in that square are copies of the single edge from $G([1235476])$, while the braids are copies of the single braid edge in $G([3214567])$. As we shuffle the elements 4 and 6 to the right, notice that we do not have edges inherited from $G([1235476])$ or $G([3214567])$ anymore. We have the vertex set of $G([1235476]) \times G([3214567])$, but no internal edges.

We would continue moving forward in this manner, sometimes with those internal edges present, but most of the time they are not.
\end{example}

There is no reason that we need to stop at splitting $\sigma$ into two pieces to work with shuffles, as we see in the following theorem:

\begin{theorem}\label{thm:spliting_perms}
    Let $\sigma \in \Sym_n$, such that $\Des(\sigma) = \cup_{i=1}^m D_i$ where the $D_i$'s are the maximal subsets of $\Des(\sigma)$ containing only consecutive elements. Then
    \begin{enumerate}
        \item There exist elements $\beta_i$ such that $\Des(\beta_i) = D_i$ for $1\leq i \leq m$, $\beta_i = w_0^{(k_i,a_i)}$ for some $k_i,a_i\in\N $, and $\sigma = \alpha \beta_1 \ldots \beta_m$.
        \item For any $u\in \calR(\alpha)$ and selection $v_i\in \calR(\beta_i)$, there exists a reduced word of $\sigma$ of the form $uv_1\ldots v_m$. 
        \item We can fully construct $\calR(\sigma)$ by looking at all the possible ways that the letters of $uv_1\ldots v_m$ can shuffle through each other.
    \end{enumerate}
\end{theorem}

\begin{proof}
Note that in the case where $\Des(\sigma) = D_1$ or $\Des(\sigma)= D_1\cup D_2$, we simply recover the results in Lemma \ref{breakingPerms} and  Lemma \ref{permsSuffle}. 

In the case of larger unions, (1) and (2) are inductive extension of the result in Lemma \ref{breakingPerms}, and (3) is an inductive extension of the result in Lemma \ref{permsSuffle}.
\end{proof}

\subsection{Characterizations of simple shuffle graphs}

There were distinct advantages to the subgraphs in Section \ref{sec:subgraphs_des}. Every edge is accounted for, and the partitions of $V(G(\sigma))$ produce induced subgraphs that are isomorphic to $G(\tau)$ for $\tau \leq \sigma$. As we saw in the first part of Section \ref{subsec:shuffles}, partitions of $V(G(\sigma))$ into subsets related to shuffles often missing the edges present in $G(\alpha)$. 

Our motivating question has to do with braid edges in graphs $G(\sigma)$. As we see in Section \ref{sec:motivation}, we can make definitive statements about braid edges if $\sigma = w_0^{(k,i)}$ for some $k,i\in \N$. The shuffle graphs help us divide up $G(\sigma)$ into more parts related to the $w_0^{(k,i)}$'s. This helps us make certain arguments, but also leads to some mess.

\begin{definition}\label{config}
Let $\alpha, \beta_1, \ldots \beta_m, \sigma\in \mathfrak{S}_n$, with $a\in \calR(\alpha)$, $b_i\in \calR(\beta_i)$ for all $1\leq i \leq m$, and $ab_1\ldots b_m\in \calR(\sigma)$ as described in Theorem \ref{thm:spliting_perms}. 

Let $C_I$ denote all the words in $\calR(\sigma)$ where the letters from $a\in \calR(\alpha)$ sit in the positions contained in $I$, where $I \subset [\ell(\sigma)]$ and $|I|=\ell(\alpha)$. We call this a \JEem{configuration}, $C_I$.

In the case that $\alpha = e$, we consider $C_I$ to be formed from all words where the letters from $b_1\in \calR(\beta_1)$ sit in the positions contained in $I$.
\end{definition}

After a shuffle of type 3, we would consider sets labeled $C_I'$. We would still be looking at where the red letters descended from $a$ (or $b_1$) would sit, but we would want to be able to differentiate this set from $C_I$. 

\begin{definition}\label{shufflesubgraph}
Let $\alpha, \beta_1, \ldots \beta_m, \sigma\in \mathfrak{S}_n$, with $u\in \calR(\alpha)$, $v_i\in \calR(\beta_i)$ for all $1\leq i \leq m$, and $uv_1\ldots v_m\in \calR(\sigma)$ as described in Theorem \ref{thm:spliting_perms}. We define $H_I$ to be the induced subgraph of $G$ with the vertex set $C_I$. 
\end{definition}

If there is a commutation move or braid move between letters of the words $a_1, a_2 \in \calR(\alpha)$, then there is an edge between all vertices of the form $a_1(b_1\ldots b_m)$ and $a_2(b_1\ldots b_m)$, for any selection of words $b_i\in \calR(\beta_i)$. Similarly, in a collection or reduced words $C_I$, any time there is a commutation or braid move between $a_1, a_2 \in \calR(\alpha)$, and the letters used in these particular moves are sitting in consecutive spots in the words contained in $C_I$, there is an edge between those vertices for any selection $b_i\in \calR(\beta_i)$.
 
Note that not all configurations result in a subgraph with $|\calR(\alpha)||\calR(\beta)|$ vertices. This is because not all letters of a reduced word of $\alpha$ need to commute with all letters of a word of $\beta$. See Example \ref{blockededges}, Figure \ref{fig:C4E1}. Also note that  we rarely have subgraphs that are isomorphic to $G(\alpha) \times G(\beta)$, as the edges are missing in many cases. See Example \ref{blockededges}, Figure \ref{fig:C4E0}.

For the moment let us consider permutations $\sigma = \alpha \beta \in \mathfrak{S}_n$ where $uv\in \calR(\sigma)$ is such that we only have shuffles of type 1 and 2 between reduced words of $\alpha$ and $\beta$.

\begin{lemma}\label{filteringlemma}
Let $\alpha, \beta, \sigma\in \mathfrak{S}_n$, with $u\in \calR(\alpha)$, $v\in \calR(\beta)$ and $uv\in \calR(\sigma)$ as described in Lemma \ref{breakingPerms} and Definition \ref{shuffledef}. Further suppose that for all $s_i$ in the support of $\alpha$ and all $s_j$ in the support of $\beta$, $|i-j|\geq 1$. That is, there are not any shuffles of the type 3 from Definition \ref{shuffledef}. Then
\[
|\calR(\sigma)|\leq |\calR(\alpha)||\calR(\beta)|\binom{\ell(\sigma)}{\ell(\alpha)}.
\]
\end{lemma}

\begin{proof}
We assume that there are not any shuffles of type 3 in $\calR(\sigma)$. So we only need to consider whether letters from $u\in \calR(\alpha)$ commute with letters in $v\in \calR(\beta)$, or not.

Let $C_I$ be defined as in Definition \ref{config}, where we have chosen the set $I$ such that $|I|=\ell(\alpha)$ and $I\subset \ell(\sigma)$. There are $\binom{\ell(\sigma)}{\ell(\alpha)}$ choices for the set $I$. Furthermore, because there are no shuffles of type 3, we have
\[
\calR(\sigma)= \bigcup_{I\subset [\ell(\sigma)], ~|I|=\ell(\alpha)}C_I.
\]
This is a disjoint union, and some sets $C_I$ could be empty.

Let us define new sets $A_I$ as follows: 
\[
A_I = \{w\mid a\in \calR(\alpha) \mbox{ in positions in }I, b\in \calR(\beta) \mbox{ in the remaining positions} \}.
\]
Each of the sets $A_I$ has size $|A_I|=|\calR(\alpha)||\calR(\beta)|$. 

These sets may contain words which are not in $\calR(\sigma)$. Let $a=a_1\ldots a_{\ell(\alpha)} \in \calR(\alpha)$, and $b=b_1\ldots b_{\ell(\beta)}\in \calR(\beta)$. Suppose that there are letters in these words $a_i$ and $b_j$ such that $|a_i-b_j|=1$. Then any index set $I$ that places the letter $a_i$ to the right of letter $b_j$ means that $A_I$ contains words that are not in $\calR(\sigma)$.

For any index set $I$, $C_I\subset A_I$ so that $|C_I|\leq |A_I|$. There is no way for $C_I$ to be larger than $A_I$, since $A_I$ already contains all possible words formed from $a\in \calR(\alpha)$ in positions in $I$. 

If $C_I=A_I$ for all index sets $I$, then 
\[
|\calR(\sigma)|= |\calR(\alpha)||\calR(\beta)|\binom{\ell(\sigma)}{\ell(\alpha)}.
\]

If there are letters in a word $a\in \calR(\alpha)$ that do not commute with letters in $b\in \calR(\beta)$, then there is some index set $J$ such that $C_J\neq A_J$. In which case,
\[
|\calR(\sigma)|< |\calR(\alpha)||\calR(\beta)|\binom{\ell(\sigma)}{\ell(\alpha)}.
\]
Therefore, we have the desired inequality.
\end{proof}

The shuffling process and configuration subgraphs that are described in the lemmas and definition above gives us a way to divide our graphs up into copies of $G(\alpha)$ and $G(\beta)$. 

\begin{corollary}\label{cor:nice_shuffle_graphs}
    Let $\alpha, \beta, \sigma\in \mathfrak{S}_n$, with $u\in \calR(\alpha)$, $v\in \calR(\beta)$ and $uv\in \calR(\sigma)$ as described in Lemma \ref{breakingPerms} and Definition \ref{shuffledef}. 

    If we arrive at $C_I$ using only shuffles of type 1 and 2, then $H_I$ is a subgraph of $G(\alpha)\times G(\beta)$. $H_I$ does not need to be an induced subgraph. If a shuffle of type 1 from Definition \ref{shuffledef} exists in this configuration, then each of those edges from $G(\alpha)$ is replaced with $|\calR(\beta)|$ copies of the same edge. 
\end{corollary}

\begin{lemma}\label{shufflelemma}
Let $\alpha, \beta, \sigma\in \mathfrak{S}_n$, with $u\in \calR(\alpha)$, $v\in \calR(\beta)$ and $uv\in \calR(\sigma)$ as described in Lemma \ref{breakingPerms} and Definition \ref{shuffledef}. 

If we arrive at $C_I$ using all three types of shuffles, $H_I$ is not necessarily a subgraph of $G(\alpha)\times G(\beta)$, because a type 3 shuffle changes $\alpha$ and $\beta$. 
\end{lemma}

\begin{proof}
    As soon as a shuffle of type 3 has occurred, we no longer have a reduced word of the form $uv$ for $u\in \calR(\alpha)$ and $v\in \calR(\beta)$. Instead, we now have $ab$ where $a\in \calR(\alpha')$, and $\calR(\beta')$ where $\sigma = \alpha'\beta'$. The subgraphs with the vertices $ab$ are not required to have any relationship to vertices in $G(\alpha)\times G(\beta)$.
\end{proof}

The reason we want to look at this family of subgraphs is because we know a lot about $G(w_0^{(k,i)})$. However, for an arbitrary $\sigma$, we have less information about $G(\sigma s_j)$ for $j\in \Des(\sigma)$. 

\begin{lemma}\label{fullycombraid}
Let $\alpha, \beta \in \mathfrak{S}_n$. Then
\[
\sum_{v\in G(\alpha) \times G(\beta)}d_\calB(v) = |\calR(\alpha)|\cdot \sum_{u\in G(\beta) }d_\calB(u) + |\calR(\beta)|\cdot \sum_{t\in G(\alpha) }d_\calB(t) .
\]
\end{lemma}

\begin{proof}
We consider $G(\alpha)\times G(\beta)$ as follows: let $v\in G(\alpha)\times G(\beta)$ be labeled as $(a,b)$ for $a\in \calR(\alpha)$ and $b\in \calR(\beta)$.

There is a braid edge between $(a,b)$ and $(a',b')$ if and only if there is a braid move between $a$ and $a'$ and $b=b'$, or if there is a braid move between $b$ and $b'$ and $a=a'$.

Therefore, the braid degree of $v\in G(\alpha)\times G(\beta)$ depends on the braid degrees of $a\in G(\alpha)$ and $b\in G(\beta)$. That is, $v=(a,b)$ is such that $d_\calB(v) =d_b^{G(\alpha)}(a) + d_b^{G(\beta)}(b) $. In order to calculate $\sum_{v\in G(\alpha) \times G(\beta)}d_\calB(v)$, we have to be able to vary over $a\in \calR(\alpha)$ and $b\in \calR(\beta)$.

Therefore, our calculation are
\begin{eqnarray*}
\sum_{v\in G(\alpha) \times G(\beta)}d_\calB(v) &=& \sum_{t\in G(\alpha)} \sum_{u\in G(\beta) } d_\calB(t) +d_\calB(u) \\
&=& \sum_{t\in G(\alpha)} \left ( \sum_{u\in G(\beta) } d_\calB(t) + \sum_{u\in G(\beta) }d_\calB(u) \right )\\
&=&  \sum_{t\in G(\alpha)} \left ( |\calR(\beta)| d_\calB(t) + \sum_{u\in G(\beta) }d_\calB(u) \right )\\
&=& \sum_{t\in G(\alpha)}|\calR(\beta)| d_\calB(t) +  \sum_{t\in G(\alpha)}\sum_{u\in G(\beta) }d_\calB(u)\\
&=& |\calR(\beta)|\cdot \sum_{t\in G(\alpha)} d_\calB(t) +  |\calR(\alpha)| \cdot \sum_{u\in G(\beta) }d_\calB(u),
\end{eqnarray*}
as desired
\end{proof}

\section{Bounds on the number of Braid Classes}\label{sec:motivation}

In this section, we heavily rely on two results we discussed in Section \ref{sec:background}: the main result from Fishel, Milićević, Patrias, and Tenner (Theorem \ref{bound1}\cite{F}), and the result from Reiner (Theorem \ref{ReinerBraid} \cite{R}).

After extensive work with examples for $n=4,5,6,7,8$, using Sage to help find the sizes of the sets of reduced words, we arrived at the following conjectures.

\begin{conjecture}\label{braidclassratio}
For all $\sigma\in \mathfrak{S}_n$,
\[
 \dfrac{1}{2} |\calR(\sigma)| \leq |\calB(\sigma )|\leq |\calR(\sigma)|.
\]
\end{conjecture}

\begin{conjecture}\label{commclassratio}
For all $\sigma\in \mathfrak{S}_n$, 
\[
0\leq |\calC(\sigma)| \leq \dfrac{1}{2}|\calR(\sigma)| + 1.
\]
\end{conjecture}

\subsection{The weak order lattice and the longest permutation}
Much like the other research done by Tenner \cite{T2} and Schilling et al. \cite{Sch}, we start with results related to $G(w_0)$.

\begin{theorem}\label{weakerV}
For $w_0\in \mathfrak{S}_n$, $|\calB(w_0)|\geq \dfrac{1}{2}|\calR(w_0)|-1$. 
\end{theorem}

\begin{proof}
We consider the graph $G_b'(w_0)$, as defined in Definition \ref{def:com_and_br_graphs}. Recall that the connected components of this graph are the braid classes of $G(w_0)$. We would like to show that there are at least $ \dfrac{1}{2}|\calR(w_0)|-1$ connected components. 

Suppose to the contrary that there are less than $ \dfrac{1}{2}|\calR(w_0)|-1$ connected components in $G'_\calB(w_0)$. Then if $|\calR(w_0)|$ is even, we have at most $\dfrac{1}{2}|\calR(w_0)|-2$ components, and if $ |\calR(w_0)|$ is odd, then there are at most $ \dfrac{1}{2}|\calR(w_0)| -\dfrac{3}{2}$ components.  


We can partition our vertex set over the connected components. Suppose there are $m$ connected components, and let $V_i$ be the vertex set for one of these components, for $1\leq i\leq m<\dfrac{1}{2}|\calR(w_0)|-1$. Then, 
\[
V(G_b'(w_0)) = \bigcup_{1\leq i \leq m} V_i.
\]

We note that a graph with $V$ vertices and $k$ components must have at least $V-k$ edges, since a component with $V_i$ vertices must have a spanning tree with $V_i-1$ edges. We assume there are $m$ components, where $m\leq \dfrac{1}{2}|\calR(w_0)|-i$, for $i\in \left \{ \dfrac{3}{2}, 2 \right \}$. Thus, we can subtract off a maximum of $\dfrac{1}{2}|\calR(w_0)|-i$ edges.

We also see that each of the connected components in $G_b'(w_0)$ are at minimum tree graphs, which means we have at least
\[
\sum_{v\in V(G(w_0))} d_{G'_\calB(w_0)}(v) \geq 2|\calR(w_0)|-2.
\]

Thus, for $i\in \left \{\dfrac{3}{2}, 2 \right \}$, we have
\begin{eqnarray*}
\sum_{v\in V(G(w_0))} d_{G'_\calB(w_0)}(v) &=& 2E(G'_\calB(\sigma))\\
&\geq & 2|\calR(w_0)|-2 - 2\left [ \dfrac{1}{2}|\calR(w_0)| - i \right ] \\
&>& |\calR(w_0)|,
\end{eqnarray*}
for either choice of $i$.

This contradicts Reiner's result in Theorem \ref{ReinerBraid} that states that the sum of the braid degrees exactly equals the size of the set of reduced words. 
\end{proof}

Initially, we may only note that this theorem is true for $w_0\in \mathfrak{S}_n$. However, in Section \ref{sec:background}, Definition \ref{w0equiv}, we defined the family of $w_0^{(k,i)}$ permutations. Since these are permutations that also produce graphs of the form $G(w_0)$ for $w_0\in \mathfrak{S}_k$, the result is true for more than one permutation in $\mathfrak{S}_n$ when $n>3$.

\subsection{Bounds on the number of braid classes for certain families of permutations}

We now use our results from Sections \ref{sec:subgraphs_des} and \ref{sec:shuffle_sub} to prove results similar to Theorem \ref{weakerV} for more general families of permutations.

In the following lemma, we refer to certain $\alpha\in \Sym_n$ as being \textit{completely commutative} permutations. In this case we do not mean fully commutative, we simply mean that a reduced decomposition of $\alpha$ is formed from commuting generators. For example, $\sigma = s_1s_3s_5$ would be completely commutative, while $\tau = s_4 s_1s_2s_3$ would be fully commutative.

\begin{lemma}\label{casesproved}
Suppose that $\Des(\sigma) =  [i,i+k-1]\cup J$, where $k$ is the length of the longest string of consecutive elements in $\Des(\sigma)$. Then 
\begin{enumerate}
\item If $k=1$,
\[
A(\sigma) := |\calR(\sigma)|-\left (\sum_{v\in G(\sigma)}d_\calB(v) \right )>0.
\]
\item If $J=\{ j\}$ and $\sigma =s_j w_0^{(k,i)}$, then $A(\sigma)>0$.
\item If $J \neq \emptyset$, $\alpha$ is completely commutative with $\Des(\alpha) = J$, and $\sigma = \alpha w_0^{(k,i)}$, then $A(\sigma)>0$.
\item If $J \neq \emptyset$, $\alpha$ is fully commutative with $\Des(\alpha) = J$,  every $s_i$ in the support of $\alpha$ commutes with all $s_i$'s in the support of $w_0^{(k,i)}$, and $\sigma = \alpha w_0^{(k,i)}$, then $A(\sigma)>0$.
\item If $J=\emptyset$ and $\sigma = w_0^{(k,i)}$, then $A(\sigma)=0$.
\end{enumerate}
\end{lemma}

\begin{proof}
First, we note that point five above is exactly Theorem \ref{ReinerBraid}. Thus we focus on proving the other three parts of our lemma.

Let $\sigma \in \mathfrak{S}_n$ where $l(\sigma) = l$ and $\Des(\sigma) =   [i,i+k-1] \cup J$ for some $k,i\in \mathbb{Z}$ such that no string of consecutive elements in $J$ has length greater than $k$. 

Suppose that for all $\tau \in \mathfrak{S}_n$ where $l(\tau)<l$, exactly one of the following is true:
\begin{enumerate}
\item $\tau = w_0^{(k,i)}$, and $A(\tau) = 0$,
\item $\tau \neq w_0^{(k,i)}$, and $A(\tau)>0$.
\end{enumerate}
We want to show that $A(\sigma)>0$ as well.

\textbf{Case 1:} Let $k=1$. That is, suppose that $\Des(\sigma)$ does not contain any consecutive elements. Then we select any $j\in \Des(\sigma)$, and consider the word $aj\in \calR(\sigma)$ where $a\in \calR(\sigma s_j)$.

If $\Des(\sigma) = \{j\}$, we note that $|\calR(\sigma)|=|\calR(\sigma s_j)|$. Then $A(\sigma) = A(\sigma s_j)>0$ by our induction hypothesis.

If $|\Des(\sigma)|>1$, Proposition \ref{w0andchildren} notes that at most one $j\in \Des(\sigma)$ produces $\sigma s_j = w_0^{(k,i)}$. Using Corollary \ref{recDEG}, we have
\begin{eqnarray*}
\sum_{v\in G(\sigma)} d_\calB(v) &=& \left ( \sum_{x\in \Des(\sigma)}\sum_{u\in G(\sigma s_x)}d_\calB(u) \right )+ 2 \cdot \sum_{x,x+1\in \Des(\sigma)}|\calR(\sigma s_xs_{x+1}s_x)| \\
&~& \\
&=& \sum_{x\in \Des(\sigma)}\sum_{u\in G(\sigma s_x)}d_\calB(u) ,
\end{eqnarray*}
because there are no consecutive elements in the set $\Des(\sigma)$.

We assume that $A(\sigma s_x)=0$ for at most one descent $x$, and $A(\sigma s_y)>0$ for all other descents $y$. There are at least two elements in $\Des(\sigma)$, and therefore,
\begin{eqnarray*}
\sum_{v\in G(\sigma)} d_\calB(v) &=& \sum_{x\in \Des(\sigma)}\sum_{u\in G(\sigma s_x)}d_\calB(u) \\
&< &  \sum_{x\in \Des(\sigma)}\left ( |\calR(\sigma s_x)|\right )\\
&=& |\calR(\sigma)|,
\end{eqnarray*}
and therefore, $A(\sigma)>0$.


\textbf{Case 2:} Let $\ell(\alpha) = 1$. Suppose that $\alpha = s_j$, $\beta =  w_0^{(k,i)}$, so that $j v \in \calR(\sigma)$ for $v\in \calR(w_0^{(k,i)})$. Because $\Des(\sigma) = [i,i+k-1] \cup \{j\}$, we know that $j\notin \{i-1, i+k\}$. Then $j$ commutes with every letter in $u\in \calR(\beta)$, so we can split up $G(\sigma)$ into subgraphs $H_t$ as defined in Definition \ref{shufflesubgraph}, where $1\leq t\leq \ell(\beta)+1$.

Since $j$ commutes with everything in $u\in \calR(\beta)$, we know that $j$ cannot be used in a braid move, and that $|V(H_t)| = |\calR(\beta)|$. Thus each of the $H_t$'s is joined to another $H_{t-1}$ and $H_{t+1}$ by commutation edges.

In any three consecutive positions in $\beta$, there are $\sum_{m,m+1 \in \Des(\beta)}|\calR(\beta s_ms_{m-1}s_m)|$ braid moves, corresponding to braid edges. Since $\beta$ is a $w_0^{(k,i)}$, using notation from Definition \ref{def:com_and_br_graphs}, we also know that 
\[
|\calR(\beta)| = 2\cdot |E_\calB(\beta)| = \sum_{v\in G(\beta)}d_\calB(v),
\]
which means that
\[
|\calR(\beta)|=2\cdot \left ( (l(\beta) -2) \sum_{m,m+1 \in \Des(\beta)}|\calR(\beta s_ms_{m-1}s_m)|\right ) = \sum_{v\in G(\beta)} d_\calB(v).
\]

For $t=1, l(\beta)+1$, the subgraph $H_t$ has exactly $|\calR(\beta)|$ vertices, and $\sum_{v\in H_t}d_\calB(v) = |V(H_t)|$.

For $t=2, l(\beta)$, the subgraph $H_t$ has exactly $|\calR(\beta)|$ vertices, but 
\[
\sum_{v\in H_t} d_\calB(v) = 2\cdot \left ( (l(\beta) -3) \sum_{m,m+1 \in \Des(\beta)}|\calR(\beta s_ms_{m-1}s_m)|\right ) <|\calR(\beta)|.
\]

For all other $3\leq t \leq l(\beta)-1$, the subgraph $H_t$ has exactly $|\calR(\beta)|$ vertices, but 
\[
\sum_{v\in H_t} d_\calB(v) = 2\cdot \left ( (l(\beta) -4) \sum_{m,m+1 \in \Des(\beta)}|\calR(\beta s_ms_{m-1}s_m)|\right ) <|\calR(\beta)|.
\]

Thus 
\[
\sum_{v\in G(\sigma)} d_\calB(v) < |\calR(\sigma)|,
\]
so that $A(\sigma)>0$ as desired.


\textbf{Case 3 and 4:} Let $\ell(\alpha) = m$. Let $\alpha = s_{a_1} \ldots s_{a_m}$, $\beta =  w_0^{(k,i)}$, so that $a_1\ldots a_m v \in \calR(\sigma)$ for $v\in \calR(w_0^{(k,i)})$. Because $\Des(\sigma) = [i,i+k-1] \cup J$, we know that $a_j\notin \{i-1, i+k\}$ for any $1\leq j \leq m$. Then $a_j$ commutes with every letter in $u\in \calR(\beta)$, so we can split up $G(\sigma)$ into subgraphs $H_t$ as defined in Definition \ref{shufflesubgraph}, where $1\leq t\leq \binom{\ell (\alpha) +\ell (\beta )}{\ell(\alpha)}$.

Since each $a_j$ commutes with everything in $u\in \calR(\beta)$, and $\alpha$ is fully commutative, we know that each $a_j$ cannot be used in a braid move, and that $|V(H_t)| = |\calR(\alpha)|\cdot|\calR(\beta)|$. Thus each of the $H_t$'s is joined to another $H_{t-1}$ and $H_{t+1}$ by commutation edges. Additionally, each of the graphs $H_t$ are a subgraph of $G(\alpha)\times G(\beta)$, where $V(H_t) = V(G(\alpha)\times G(\beta))$, and $E(H_t)\subset E(G(\alpha)\times G(\beta))$. We also know from \ref{fullycombraid} that the only braid edges are in the copies of $G(\beta)$.

Using the same argument as Case 2, we have the following facts about our subgraphs $H_t$: 

For any subgraph $H_t$ where the letters of $\beta$ sit in consecutive positions, $H_t$ has $|\calR(\alpha)||\calR(\beta)|$ vertices, and $\sum_{v\in H_t}d_\calB(v) = |V(H_t)|$. 

For any subgraph $H_t$ where at least two letters of $\beta$ are separated by at least one letter of $\alpha$, $H_t$ still has exactly $|\calR(\alpha)||\calR(\beta)|$ vertices, but $ \sum_{v\in H_t} d_\calB(v)  <|\calR(\alpha)||\calR(\beta)|$
Thus 
\[
\sum_{v\in G(\sigma)} d_\calB(v) < |\calR(\sigma)|,
\]
so that $A(\sigma)>0$ as desired.
\end{proof}

Note that the cases we have proven in Lemma \ref{casesproved} relied heavily on the work we did in Sections \ref{sec:subgraphs_des} and \ref{sec:shuffle_sub}. This Lemma also covers all the cases we have been able to prove.


\begin{theorem}\label{mainresult}
Let $\sigma \in \mathfrak{S}_n$ be one of the permutations covered by Lemma \ref{casesproved}. Then 
\[|\calB(\sigma)| \geq \dfrac{1}{2}|\calR(\sigma)|-1 \quad \mbox{ and } \quad |\calC(\sigma)| \leq \dfrac{1}{2}|\calR(\sigma)| +2.
\]
\end{theorem}

\begin{proof}
We can follow all the steps of the proof of Lemma \ref{weakerV}, replacing $w_0$ with $\sigma$. The proof is dependent on arriving at 
\[
\sum_{v\in G(w_0)} d_\calB(v) > |\calR(w_0)|,
\]
which is a contradiction for $w_0$. If $\sigma$ is one of the permutations in Lemma \ref{casesproved}, then $A(\sigma)>0$, and this is also a contradiction for $\sigma$. 

Thus $|\calB(\sigma)| \geq \dfrac{1}{2} |\calR(\sigma)| -1$. By Theorem \ref{bound1}, 
\[
|\calB(\sigma)|+|\calC(\sigma)|-1 \leq |\calR(\sigma)|.
\]
We can now rewrite this as 
\[
\dfrac{1}{2}|\calR(\sigma)|-1+|\calC(\sigma)|-1 \leq |\calR(\sigma)|,
\]
which means that
\[
|\calC(\sigma)| \leq \dfrac{1}{2}|\calR(\sigma)| +2,
\]
as desired.
\end{proof}

While not complete for arbitrary permutations, this theorem represents all current knowledge of the relationship between $|\calB(\sigma)|$ and $|\calR(\sigma)|$. We used all of our new tools related to subgraphs to arrive at this result, and believe that additional work on our new families of subgraphs are required to prove any more cases.

\section{Future Work}\label{sec:future}
Finally, we detail some future directions of study for the work done in this paper.

\begin{enumerate}
\item There are two possible results that would prove Theorem \ref{mainresult} for all permutations:

\begin{conjecture}\label{path1}
    For any $\sigma \in \Sym_n$, \[
    A(\sigma)= |\calR(\sigma)|-\left (\sum_{v\in G(\sigma)}d_\calB(v) \right ) \geq 0.
    \]
\end{conjecture}

\begin{conjecture}\label{path2}
    For all $\sigma \in \Sym_n$, when $\sigma = w_0^{(k,i)}$ for some $k,i\in \N$, then $G(\sigma)$ has the highest proportion of braid edges to vertices in any graph of $\calR(\sigma)$ in $\Sym_n$.
\end{conjecture}

 We have spent considerable time attempting to prove these conjectures, but as of the writing of this paper, the cases present in Lemma \ref{casesproved} represents all of our progress on Theorem \ref{mainresult}. 

 \item Theorem \ref{thm:spliting_perms} allows us to break permutations into more parts than we consider in Lemma \ref{casesproved}. That is, more parts that are equivalent to $w_0^{(k,i)}$'s. Can we add cases to Lemma \ref{casesproved} using a more general division of $\sigma$?

\item If Theorem \ref{mainresult} can be proven for all permutations, we still have Conjectures \ref{braidclassratio} and \ref{commclassratio}, which we believe are the true limits of these ratios. Are there other methods of proving these conjectures without our work on subgraphs?

\item Since these families of subgraphs have not been studied, we intend to continue to investigate what they can tell us about $\calR(\sigma)$ apart from congruence classes and number of edges in $G(\sigma)$.

\item As noted in Section \ref{subsec:shuffles}, 12 and 21 inflations have been well studied. Since shuffles are a more general method of looking at permutations, we continue to study what they can tell us about $\calR(\sigma)$ and $G(\sigma)$.

\end{enumerate}

\section*{Acknowledgements}

Most of this research was done as part of the author's PhD dissertation at Arizona State University. The author would like to thank her advisor, Susanna Fishel, for advice and support related to the writing of this document. The author would like to thank J. Carlos Mart\'inez Mori for the development of Figure \ref{fig:S4}. The author would also like to thank Samantha Dahlberg and Bridget Tenner for answering questions on related problems in research on sets of reduced words.

\bibliographystyle{plain}
\bibliography{bibliography}

\end{document}